\documentclass[12pt]{amsart}
 \usepackage{mathtools}
\mathtoolsset{showonlyrefs,showmanualtags} 

\usepackage{amssymb, amsmath, amsthm, esint}
\usepackage{mathrsfs}
\usepackage{bbm,dsfont}
\usepackage{hyperref}
\usepackage{mathtools}
\usepackage{fullpage}
\usepackage{color}

\usepackage[backrefs]{amsrefs}

 \usepackage[top=1.5in, bottom=1.5in, left=1.25in, right=1.25in]	{geometry}
\usepackage[all]{xy}
\usepackage{amsmath}
\usepackage{amssymb}
\usepackage{amsthm}
\usepackage{amscd}

\numberwithin{equation}{section}

\newtheorem{conj}{Conjecture}
\newtheorem{theorem}[equation]{Theorem}
\newtheorem{definition}[equation]{Definition}
\newtheorem{proposition}[equation]{Proposition}
\newtheorem{cor}[equation]{Corollary}
\newtheorem{lemma}[equation]{Lemma}

\usepackage{hyperref} 
\hypersetup{
    colorlinks=true,       
    linkcolor=blue,          
    citecolor=magenta,        
    filecolor=magenta,      
    urlcolor=cyan           
}

\begin{document}
\title{Quantitative Convergence for Sparse Ergodic Averages in $L^1$}

\author{Ben Krause}
\address{
Department of Mathematics,
University of Bristol \\
Woodland Rd, Bristol BS8 1UG}
\email{ben.krause@bristol.ac.uk}

\author{Yu-Chen Sun}
\address{
Department of Mathematics,
University of Bristol \\
Woodland Rd, Bristol BS8 1UG}
\email{yuchensun93@163.com}

\date{\today}

\maketitle

\begin{abstract}
We provide a unified framework to proving pointwise convergence of sparse sequences, deterministic and random, at the $L^1(X)$ endpoint. Specifically, suppose that
\[ a_n \in \{ \lfloor n^c \rfloor, \min\{ k : \sum_{j \leq k} X_j = n\} \} \]
where $X_j$ are Bernoulli random variables with expectations $\mathbb{E} X_j = j^{-\alpha}$, and we restrict to $1 < c < 7/6, \ 0 < \alpha < 1/2$.

Then (almost surely) for any measure-preserving system, $(X,\mu,T)$, and any $f \in L^1(X)$, the ergodic averages
\[ \frac{1}{N} \sum_{n \leq N} T^{a_n} f \]
converge $\mu$-a.e. Moreover, our proof gives new quantitative estimates on the rate of convergence, using jump-counting/variation/oscillation technology, pioneered by Bourgain.

This improves on previous work of Urban-Zienkiewicz and Mirek, who established the above with $c = \frac{1001}{1000}, \ \frac{30}{29}$, respectively, and LaVictoire, who established the random result, all in a non-quantitative setting.

\end{abstract}

\section{Introduction}

The topic of this paper is quantitative convergence of ergodic averages. We will be concerned, in particular, with the issue of ergodic averages along sparse times at the $L^1(X)$ endpoint, a topic which grew out of a conjecture of Rosenblatt-Wierdl \cite[Conjecture 4.1]{RW}.

\begin{conj}
    Suppose that $\{a_n\}$ has zero Upper Banach Density:
    \[ \limsup_{|I| \to \infty \text{ an interval}} \frac{|I \cap \{ a_n \}|}{|I|} = 0.\]
Then for any probability space, $(X,\mu)$, equipped with an aperiodic, measure-preserving transformation $T:X \to X$, there exists $f \in L^1(X)$ so that
\[ \frac{1}{N} \sum_{n \leq N} T^{a_n} f, \; \; \; T^k f(x) := f(T^k x), \]
does \emph{not} converge almost everywhere.
    \end{conj}

This conjecture was disproven by Buczolich \cite{Buc}; subsequently, Urban and Zienkiewicz \cite{UZ} proved that for any $\sigma$-finite measure-space $(X,\mu)$, equipped with a measure-preserving transformation, $T: X \to X$, and any $f \in L^1(X)$, the ergodic averages 
\[ \frac{1}{N} \sum_{n \leq N} T^{a_n} f, \; \; \; a_n = \lfloor n^c \rfloor, \ 1 < c < \frac{1001}{1000}, \]
converge almost everywhere, with the $c = 1$ case appearing as the classical Birkhoff Pointwise Ergodic Theorem \cite{BI}.

For brevity, for the remainder of this paper, we will refer to such $\sigma$-finite measure spaces $(X,\mu)$, equipped with measure-preserving transformations $T:X \to X$, as \emph{measure-preserving systems}. And, we will say that a sequence $\{ a_n \}$ is \emph{universally $L^1$-good} if for every measure-preserving system, $(X,\mu,T)$, and every $f \in L^1(X)$, the ergodic averages
\[ \frac{1}{N} \sum_{n \leq N} T^{a_n} f \]
converge $\mu$-a.e.

With this notation in mind, the results of \cite{UZ} were extended by Mirek \cite{M} to $\lfloor n^c \rfloor, \ 1 < c < 30/29 = 1.03\dots$ and certain perturbations of these sequences, see also \cite{D}. Before Mirek's work, LaVictoire \cite{L} used the probabilistic method to prove that, generically, sequences with density slightly greater than the squares are universally $L^1$-good, partially extending work of Bourgain \cite[\S 8]{B0} to the $L^1$-setting. More precisely, he proved that whenever $\{ X_n \}$ are independent Bernouilli Random Variables with
\[ \mathbb{E} X_n  = n^{-\alpha}, \ 0 < \alpha < 1/2,\]
and
\begin{align}\label{e:hit} a_n := \min\{ k : \sum_{j=1}^k X_j = n \}\end{align}
are random hitting times, which almost surely satisfy the approxtotic
\[ a_n \approx n^{\frac{1}{1-\alpha}},\]
by Chernoff's Inequality, Lemma \ref{l:Chern} below, see \cite[\S 5]{BOOK}, then almost surely, $\{ a_n \}$ are universally $L^1$-good. This work complemented the famous non-convergence result of Buczolich and Mauldin for ergodic averages along the squares at the $L^1$ endpoint \cite{BM}, later extended by LaVictoire to the set of primes and further monomials \cite{L1}; on the other hand, see \cite{Christ} for sparse (arithmetic) examples of sequences which are universally $L^1$-good. Indeed, a major theme of this line of inquiry is to establish convergence results for increasingly sparse sequences of times according to the absence/presence of arithmetic structure.

The goal of this paper, therefore, is two-fold: first, we extend the class of sparse deterministic sequences which are universally $L^1$-good; second, building off the approach of \cite{UZ}, we establish a unified approach to quantifying convergence of ergodic averages at the $L^1$ endpoint, which addresses LaVictoire's work in tandem with our deterministic results. To state our main results, we recall three ways to quantify oscillation, introduced to the pointwise ergodic-theoretic setting by Bourgain in \cite{B0,Bp,B2}.

\begin{definition}\label{d:def}
    For a sequence of scalars $\{ a_k \} \subset \mathbb{C}$ define the (greedy) jump-counting function at altitude $\epsilon > 0$,
    \begin{align} 
    N_\epsilon(a_k) := \sup\{ M &: \text{There exists } k_0 < k_1 < \dots < k_M \\
    & \qquad : |a_{k_n} - a_{k_{n-1}}| > \epsilon \text{ for each } 1 \leq n \leq M \}.
    \end{align}
And for each $0 < r < \infty$, define the \emph{$r$-variation} to be
\[ \mathcal{V}^r(a_k) := \sup \big( \sum_n |a_{k_n} - a_{k_{n-1}}|^r \big)^{1/r} \]
where the supremum runs over all finite increasing subsequences; we define
\[ \mathcal{V}^{\infty}(a_k) := \sup_{n \neq m} |a_n - a_m|\]
to be the \emph{diameter} of the sequence. Finally, for an increasing sequence $\{ M_j \} \subset \mathbb{N}$, we define the \emph{oscillation operator},
\[ \mathcal{O}_{\{M_j\}}(a_k) := \big( \sum_j \max_{M_j \leq k \leq M_{j+1}} |a_k - a_{M_j}|^2 \big)^{1/2}. \]
\end{definition}
To emphasize the utility of the above operators in quantifying pointwise convergence phenomena, note that the statement
\[ N_{\epsilon}(a_k) < \infty\]
for each $\epsilon >0$ is equivalent to the statement that $\{ a_k \}$ converges, as is the estimate
\[ \sup_{\{ M_j \}} \mathcal{O}_{\{ M_j\}_{j \leq J}}(a_k) = o_{J \to \infty}(J^{1/2}); \]
and the variation operators, classically used to quantify convergence in the martingale context \cite{LE}, neatly quantify convergence in that
\[ \sup_\epsilon \epsilon N_{\epsilon}(a_k)^{1/r} \leq \mathcal{V}^r(a_k), \; \; \; \sup_{\{ M_j \}} \mathcal{O}_{\{ M_j\}_{j \leq J}}(a_k) \leq J^{\max\{1/2-1/r,0\}} \cdot \mathcal{V}^r(a_k).\]

For a sequence of functions, $\{ f_N \}$, we define the jump-counting operator, the $r$-variation operator, and the oscillation operator, respectively, as 
\begin{align}\label{e:osc} N_{\epsilon}(f_N)(x) := N_{\epsilon}(f_N(x)), \; \mathcal{V}^r(f_N)(x) := \mathcal{V}^r(f_N(x)), \; 
\mathcal{O}_{\{M_j\}}(f_N)(x):=  
\mathcal{O}_{\{M_j\}}(f_N(x)).\end{align}

We now state our main results, which concern sequences 
\[ \mathbb{D} \subset \mathbb{N}\] 
which are $\lambda$-lacunary sequences, i.e. $N'/N \geq \lambda > 1$ for all $N < N' \in \mathbb{D}$.

\medskip

Our first result concerns pointwise convergence of our deterministic sequences:
\begin{theorem}\label{t:maindet}
Let $\mathbb{D}$ be $\lambda$-lacunary. Then for any $\epsilon > 0$, any $r > 2$, and any increasing sequence $\{ M_j \} \subset \mathbb{D}$, there exists an absolute constant $C_\lambda(c) < \infty$ so that the following estimate holds uniformly for each measure-preserving system,
\begin{align}
    &\| \epsilon N_\epsilon( \frac{1}{N} \sum_{n \leq N} T^{a_n} f : N \in \mathbb{D})^{1/2} \|_{L^{1,\infty}(X)} + \frac{r-2}{r} \| \mathcal{V}^r(\frac{1}{N} \sum_{n \leq N} T^{a_n} f : N \in \mathbb{D}) \|_{L^{1,\infty}(X)} \\
    & \qquad +
    \| \mathcal{O}_{\{M_k\}}(\frac{1}{N} \sum_{n \leq N} T^{a_n} f : N \in \mathbb{D}) \|_{L^{1,\infty}(X)}  \leq C_\lambda(c) \| f \|_{L^1(X)},
\end{align}
whenever $a_n = \lfloor n^c \rfloor, 1 < c < 7/6$.
\end{theorem}

Our second result addresses the random case:
\begin{theorem}\label{t:mainran}
Let $\mathbb{D}$ be $\lambda$-lacunary. Then for any $\epsilon > 0$, any $r > 2$, and any increasing sequence $\{ M_j \} \subset \mathbb{D}$, there exists an absolute constant $C_\lambda(\omega) < \infty$ so that the following estimate holds uniformly for each measure-preserving system,
\begin{align}
    &\| \epsilon N_\epsilon( \frac{1}{N} \sum_{n \leq N} T^{a_n} f : N \in \mathbb{D})^{1/2} \|_{L^{1,\infty}(X)} + \frac{r-2}{r} \| \mathcal{V}^r(\frac{1}{N} \sum_{n \leq N} T^{a_n} f : N \in \mathbb{D}) \|_{L^{1,\infty}(X)} \\
    & \qquad +
    \| \mathcal{O}_{\{M_k\}}(\frac{1}{N} \sum_{n \leq N} T^{a_n} f : N \in \mathbb{D}) \|_{L^{1,\infty}(X)}  \leq C_\lambda(\omega) \| f \|_{L^1(X)},
\end{align}
whenever $a_n = a_n(\omega)$ is as in \eqref{e:hit} with $0 < \alpha < 1/2$, away from a set of zero probability.
\end{theorem}

As a corollary, we establish the following pointwise ergodic theorems at the $L^1$ endpoint.

\begin{theorem}[Non-Quantitative Formulation]\label{t:cor}
    For any measure-preserving system, $(X,\mu,T)$, and any $f \in L^1(X)$ (almost surely)
    \[ \frac{1}{N} \sum_{n \leq N} T^{a_n} f \]
    converges $\mu$-a.e. whenever $a_n$ are as in Theorem \ref{t:maindet} or \ref{t:mainran}.
\end{theorem}

Here $7/6 \approx 1.167$ should be compared with Mirek's $30/29 \approx 1.034$. This improved range is due to an additional case analysis in analyzing the statistics of an appropriate counting function, see \eqref{e:countingfunction} below. More precisely, addressing this function reduces to understanding expressions like
\begin{align}
    \sum_{N/2 < x+m, m \leq N} e^{2 \pi i ( h_2(x+m)^{1/c} - h_1m^{1/c})}
\end{align}
where $h_1,h_2$ are restricted to certain subsets of $[1,N]$. Our argument derives from splitting into three cases according to the relative sizes of $h_1,h_2,x$ to more efficiently estimate the above; the main technical tool that allows us to do so is Lemma \ref{l:expsumh1h2u1u2}, which addresses the three relevant cases. In particular, unlike in previous work where understanding the counting function at ``small" values of $x$ produced difficulties, our arguments only degrade for ``large" values of $x \approx N$, see \eqref{e:largexbad}, where one is forced to absorb factors of $|x|^{1/2}$.



\subsection{Proof Strategy}
By Calder\'{o}n's Transference Principle \cite{C1}, see also Lemma \ref{l:trans} below, to prove Theorem \ref{t:maindet} or \ref{t:mainran} it suffices to work in a single measure-preserving system, namely the integers with the shift
\[ (X,\mu,T) = (\mathbb{Z},|\cdot|,x \mapsto x-1).\]
In this context, establishing $L^p$ estimates are fairly straightforward, deriving from Proposition \ref{p:jones} and a Fourier transform argument;\footnote{Since Lemma \ref{l:spdexpestimate} persists for all $1 < c < 2$, \eqref{e:Ehat} holds for pertaining error terms in this larger range, and we can prove $L^p$ convergence results for thicker deterministic sequences, $\lfloor n^c \rfloor, \ 1 < c < 2$, easily replacing lacunary sequences with those of the form $\{ \lfloor 2^{k^\epsilon} \rfloor : k \geq 1\}$. And, by arguing as in \cite[\S 8, Exercise 8.24]{BOOK}, dominating the short $2$-variation inside sub-lacunary blocks $\{ [2^{k^\epsilon},2^{(k+1)^\epsilon}) : k \}$ by the total variation and using the Lipschitz nature of the averaging operators
\[ N \mapsto \frac{1}{N} \sum_{n \leq N} T^{a_n} f, \]
we can remove the restriction to lacunary times -- but only in the case $p > 1$.} the main task is to establish the weak-type $(1,1)$ estimate.

The paradigm for doing so is that of Calder\'{o}n-Zygmund, which leverages four different types of arguments to push exponents down from $p=2$ to the $p=1$ endpoint:  
\begin{itemize}
    \item ``$L^0$" methods, which involve excising exceptional sets;
    \item $L^1$ methods, involving the triangle inequality;
    \item $L^2$ methods, using orthogonality considerations;
    \item $L^{\infty}$ methods, using pointwise control.
\end{itemize}
The role of $L^2$-based orthogonality methods in proving weak-type estimates was not present in the classical context, but was imported to the field by Fefferman \cite{Fef} to address ``singular" averaging operators, and figured prominently in celebrated work of Christ \cite{C00}; see also \cite{STW1,STW2} and \cite{CK} for more modern adaptations.

Our argument, built out of \cite{UZ}, makes use of all four techniques, but especially $\ell^2$-orthogonality methods, which in turn derive from additive combinatorial considerations concerning the statistics of the difference sets
\[ \{ a_n : n \leq N \} - \{ a_n : n \leq N \};\]
this already imposes a natural barrier for sequences with density like the squares, as if $\{ a_n \}$ has density comparable to the set of squares then we might expect the difference set
\[ \{ a_n : n \leq N \} - \{ a_n : n \leq N\} \subset [-N^2,N^2] \]
to have cardinality $\approx N^2$, making it very difficult to derive regularity of the counting function
\begin{align}\label{e:countingfunction} |\{ x : x = a_n - a_m : n, m \leq N \}|;\end{align}
contrast this to the case of thicker, more slowly growing sequences.

Establishing Theorem \ref{t:cor}, our pointwise convergence result, then follows by suitably transferring our main analytic result on sequence-space functions, upon applying van der Corput's method of exponential sums/concentration of measure phenomena, respectively, to show that our deterministic/random classes of examples fall into the desired paradigm.

\subsection{Acknowledgment}
We would like to thank the anonymous referee for a careful review of an earlier draft, which greatly improved the overall strength of the paper.

\section{Preliminaries}\label{ss:not}
As a first-order approximation, we let $\mathbf{N}(a_n)$ denote a homogeneous, quasi-sub-additive function, satisfying the bounds
\begin{align}\label{e:l2} \mathbf{N}(a_k) \lesssim (\sum_k |a_k|^2)^{1/2}. \end{align}
Note that all measurements of oscillation introduced in Definition \ref{d:def} are (essentially) of the form $\mathbf{N}$ whenever $r \geq 2$ in the definition of $\mathcal{V}^r$, see \eqref{e:triangle} and \eqref{e:homo} for the inequalities $\epsilon N_{\epsilon}(a_k)^{1/2}$ satisfies; more precisely, in addition to \eqref{e:l2}, $\mathbf{N}$ will need to satisfy the inequalities
\begin{align}\label{e:jump} \mathbf{N}(\sum_{l=1}^L a^{(l)}_k) \lesssim \min\big\{ L \sum_{l=1}^L \mathbf{N}_L(a_k^{(l)}), \sum_{m =1}^L m^2  \mathbf{N}_m(a_k^{(m)}) \big\} \end{align}
and
\[ \mathbf{N}(\lambda a_k) \leq |\lambda| \cdot \mathbf{N}_{|\lambda|}(a_k) \]
where each $\mathbf{N}, \mathbf{N}_l, \mathbf{N}_{|\lambda|}$ satisfy \eqref{e:l2}, as well as a common upper bound
\[ \| \mathbf{N}_*(f_k(x)) \|_{L^2(X)} \leq A, \; \; \; \mathbf{N}_* \in \{ \mathbf{N}_L, \mathbf{N}_m, \mathbf{N}_{|\lambda|} \} 
\]
whenever 
\[ 0 < \| \mathbf{N}(f_k(x)) \|_{L^2(X)} \leq A.  \]

For the remainder of the paper, we will restrict to the range $r >2$, and will reserve the character $r$ for the variation parameter.

With this formulation in mind, we can neatly express the following result concerning quantitative convergence of the Birkhoff averages \cite{J}
\[ B_N^{\phi} *f(x) := \frac{1}{N} \sum_{n \leq N} f(x-n) \phi(n/N);\]
here and throughout the remainder of the paper, we will let $\phi \in \mathcal{C}^2([-10,10])$ be smooth functions satisfying 
\begin{align}\label{e:phi} \| \phi\|_{L^{\infty}(\mathbb{R})} + \| \phi'\|_{L^{\infty}(\mathbb{R})} + \| \phi''\|_{L^{\infty}(\mathbb{R})} \leq 100.\end{align}
In practice, we will often specialize to
\begin{align}\label{e:phia} \phi_\alpha(t) := \chi(t)/t^\alpha, \; \; \; \mathbf{1}_{[1/2,2]} \leq \chi \leq \mathbf{1}_{[1/4,4]} \end{align}
where $\chi$ is smooth and $0 < \alpha \leq 1$, though we will principally be interested in the case where $0 < \alpha < 1/2$.

\begin{proposition}\label{p:jones}
    Suppose that $\mathbf{N}$ is one of the operators in Definition \ref{d:def} and $\phi$ is as in \eqref{e:phi}. Then for each $p \geq 1$ there exist constants $0 < C_p < \infty$ so that
    \begin{align} 
    &\| \mathbf{N}(B_N^{\phi}*f : N \in \mathbb{N}) \|_{\ell^{p,\infty}(\mathbb{Z})} \\
    & \leq \begin{cases} C_p \| f\|_{\ell^p(\mathbb{Z})} & \text{ if } \mathbf{N} \text{ is a jump-counting/oscillation operator} \\
    C_p \frac{r}{r-2} \| f \|_{\ell^p(\mathbb{Z})} & \text{ if } \mathbf{N} = \mathcal{V}^r, \ r > 2. \end{cases}
    \end{align}
Consequently, for any measure-preserving system $(X,\mu,T)$, the analogous bounds for 
    \[ \| \mathbf{N}(\mathbb{E}_{[N]} T^n f : N \in \mathbb{N}) \|_{L^{p,\infty}(X)} \]
    hold with the same constants.
\end{proposition}

We let $\lambda > 1$ be arbitrary, and let $\mathbb{D} = \mathbb{D}_\lambda \subset \mathbb{N}$ be a $\lambda$-lacunary sequence. For the remainder of the paper, all times will derive from $\mathbb{D}$, and all implicit constants will be allowed to depend on $\lambda$; note that no such restriction is needed for Proposition \ref{p:jones}.
\subsection{Notation}
We use
\[ e(t) := e^{2 \pi i t} \]
throughout to denote the complex exponential, and 
we let
\[ M_{\text{HL}}f(x) := \sup_{r \geq 0} \frac{1}{2r+1} \sum_{|n| \leq r} |f(x-n)|  \]
denote the Hardy-Littlewood Maximal Function. $\{ X_n \}$ will denote independent Bernoulli Random Variables. Throughout, $1 < c < 7/6$ will be a real number bounded above by $7/6$ unless otherwise indicated, and we will mostly be interested in the range $0 < \alpha < 1/2$. We let
\[ \mathbb{E}_{n \in X} a_n := \frac{1}{|X|} \sum_{n \in X} a_n, \]
set $[N] := \{1,\dots,N\}$, and let $\delta_{p}(x) := \mathbf{1}_{x = p}$ denote the point-mass at $p \in \mathbb{Z}$.

\subsection{Asymptotic Notation}\label{sss:O}
We will make use of the modified Vinogradov notation. We use $X \lesssim Y$ or $Y \gtrsim X$ to denote
the estimate $X \leq CY$ for an absolute constant $C$ and $X, Y \geq 0.$  If we need $C$ to depend on a
parameter, we shall indicate this by subscripts, thus for instance $X \lesssim_p Y$ denotes the estimate $X \leq C_p Y$ for some $C_p$ depending on $p$. We use $X \approx Y$ as shorthand for $Y \lesssim X \lesssim Y$. We use the notation $X \ll Y$ or $Y \gg X$ to denote that the implicit constant in the $\lesssim$ notation is $\geq 2^{2^{100}}$, say, and analogously $X \ll_p Y$ and $Y \gg_p X$.

We also make use of big-O and little-o notation: we let $O(Y)$  denote a quantity that is $\lesssim Y$ , and similarly
$O_p(Y )$ will denote a quantity that is $\lesssim_p Y$; we let $o_{t \to a}(Y)$
denote a quantity whose quotient with $Y$ tends to zero as $t \to a$ (possibly $\infty$).

\section{Calder\'{o}n-Zygmund Theory}

We begin by recording the following straightforward lemma, which will be useful for establishing $\ell^2$-estimates, which we use to anchor our endpoint arguments. The $\ell^p$-formulation is no more complicated, and follows from interpolating
\[ \| \mathcal{E}_N*f \|_{\ell^p(\mathbb{Z})} \leq \| \widehat{\mathcal{E}_N} \|^{2/p^*}_{L^{\infty}(\mathbb{T})} \cdot \| f \|_{\ell^p(\mathbb{Z})}, \; \; \; 1 < p < \infty, \ p^* := \max\{p,p'\}.\]

Below, we will slightly abuse notation and conflate an operator with its kernel, thus identifying e.g.\
\[ A_N f = A_N*f.\]
\begin{lemma}\label{l:tri}
    Suppose that $A_N$, $B_N$, $\mathcal{E}_N$, are convex convolution operators, with
    \[ A_N*f = B_N*f + \mathcal{E}_N*f.\]
Further, suppose that for some constant $0 < C_p < \infty$
\[ \| \mathbf{N}(B_N* f : N \in \mathbb{D}) \|_{\ell^p(\mathbb{Z})} \leq C_p \| f \|_{\ell^p(\mathbb{Z})}\]
and 
\[ \| \widehat{\mathcal{E}_N} \|_{L^{\infty}(\mathbb{T})} \lesssim N^{-\epsilon} \]
for some $\epsilon > 0$.
Then
\[ \| \mathbf{N}(A_N *f : N \in \mathbb{D}) \|_{\ell^p(\mathbb{Z})} \lesssim (C_p + \frac{p^*}{\epsilon}) \| f \|_{\ell^p(\mathbb{Z})}.\]    \end{lemma}


In what follows, for simplicity we will address the case where
\[ \mathbb{D} = 2^{\mathbb{N}};\]
to pass to the general case, one defines the \emph{dyadic scale} of an element $N \in \mathbb{D}$ to be the unique integer $n = n(N)$, so that
\begin{align}
    2^{n-1} < N \leq 2^n,
\end{align}
and uses this input instead of the identity $N = 2^n$ in the below proof; with this change in mind, the implicit constants in the below adjust by $O(\frac{1}{\lambda - 1})$.

We will deal with convolution operators, $A_N$, satisfying the following properties; below, $0 < \alpha < 1$ is fixed.

\begin{enumerate}\label{e:properties}
    \item \textbf{$\ell^2(\mathbb{Z})$-Boundedness}: $\| \mathbf{N}(A_N*f) \|_{\ell^2(\mathbb{Z})} \lesssim \| f \|_{\ell^2(\mathbb{Z})}$;
    \item \textbf{Sparse Support}: $A_N$ is supported in $[0,N]$ with
\[ |\text{supp}(A_N)|  \approx N^{1-\alpha};\]
\item \textbf{Reflection Symmetry}: With $\tilde{g}(x):=g(-x)$,
\[ A_N * \tilde{A_N} = D_N^{-1} \delta_{\{0\}} + \rho_N + O(N^{-\epsilon - 1})\]
where $D_N \approx N^{1-\alpha}, \ \epsilon > 0$, and $\rho_N(0) = 0$ is an even function which satisfies
\[ |\rho_N(x) - \rho_N(x+h)| \lesssim \frac{|h|}{N^2} \]
whenever $|x|, |x+h| \geq N^{1-\alpha}$, and $|\rho_N| \lesssim 1/N$.
\end{enumerate}

\begin{proposition}\label{p:weak}
Under the above hypotheses
\[ \| \mathbf{N}(A_N f) \|_{\ell^{1,\infty}(\mathbb{Z})} \lesssim \|f \|_{\ell^1(\mathbb{Z})}.\]
\end{proposition}
\begin{proof}
By homogeneity, possibly after replacing $\mathbf{N} \longrightarrow \mathbf{N}_{|\lambda|}$, it suffices to prove that
\[ |\{ \mathbf{N}(A_Nf) \geq 1\}| \lesssim \|f \|_{\ell^1(\mathbb{Z})}.\]

Let 
\[ X := \bigcup_{M \leq N} \{ \text{supp } A_M\} + \{ x : |f(x)| \geq N^{1-\alpha} \} \]
and
\[ E := \bigcup 100 Q,\]
where here and below, $Q$ are maximal dyadic sub-intervals inside $\{ M_{\text{HL}} f \gtrsim 1\}$, so that in particular
\begin{align}\label{e:stop} \sum_{n \in Q} |f(n)| \approx |Q|,
\end{align}
and
\begin{align}\label{e:1,1}
|E| \lesssim \sum_Q |Q| = |\{ M_{\text{HL}} f \gtrsim 1 \}| \lesssim \| f \|_{\ell^1(\mathbb{Z})}.
\end{align}
Using the trivial estimate
\begin{align}
&| \{ \text{supp } A_M\} + \{ x : |f(x)| \geq N^{1-\alpha} \} | \leq |\{ \text{supp } A_M\}| \cdot |\{ x : |f(x)| \geq N^{1-\alpha} \}| \\
& \qquad \lesssim M^{1-\alpha} |\{ x : |f(x)| \geq N^{1-\alpha}\}| \end{align}
and \eqref{e:1,1}, it suffices to prove that
\[ |\{ (X \cup E)^c : \mathbf{N}(A_N*f) \geq 1\}| \lesssim \|f \|_{\ell^1(\mathbb{Z})}.\]
For each $N = 2^n$, we decompose 
\[ f = f^{\geq n} + \sum_{s \leq n} B_s^n + g\]
where 
\[ f^{\geq n} := f \cdot \mathbf{1}_{|f| \geq 2^{(1-\alpha)n}},\]
where
\[ B_s^n = \sum_{|Q| = 2^s} b_Q^n \]
with
\[ b_Q^n := \big( f \cdot \mathbf{1}_{|f| \leq 2^{(1-\alpha)n}} - \frac{1}{|Q|} \sum_{x \in Q} f(x) \cdot \mathbf{1}_{|f| \leq 2^{(1-\alpha)n}} \big) \cdot \mathbf{1}_Q, \]
so that $\| b_Q^n \|_{\ell^1(\mathbb{Z})} \lesssim |Q|$ by \eqref{e:stop}, and $|g| \lesssim 1$ is defined by subtraction. More generally, note that
\begin{align}\label{e:set} \sup_k \| \sum_{s \leq m} B_s^k \|_{\ell^1(J)} \lesssim |E \cap J| 
\end{align}
for any $|J| \geq 2^m$.

With $\mathbf{N}_1,\mathbf{N}_2$ as in \eqref{e:jump}, we use the $\ell^2(\mathbb{Z})$ boundedness of $\mathbf{N}(A_N*g)$ to estimate
\begin{align}
    &|\{ (X \cup E)^c : \mathbf{N}(A_N*f) \geq 1\}| \lesssim |\{ \mathbf{N}_1(A_N*g) \gtrsim 1\}| + |\{ \mathbf{N}_2( A_N * \sum_{s \leq n} B_{n-s}^n : N) \gtrsim 1 \}| \\
    & \qquad \lesssim \| \mathbf{N}_1(A_N*g) \|_{\ell^2(\mathbb{Z})}^2 + |\{ \sum_N |A_N * \sum_{s \leq n} B_{n-s}^n|^2 \gtrsim 1\}| \\
    & \qquad 
    \lesssim \| g \|_{\ell^2(\mathbb{Z})}^2 + \sum_N \| A_N * \sum_{s \leq n} B_{n-s}^n\|_{\ell^2(\mathbb{Z})}^2 \\
    &  \qquad \lesssim  
    \|f \|_{\ell^1(\mathbb{Z})} + \sum_N \| A_N * \sum_{s \leq n} B_{n-s}^n\|_{\ell^2(\mathbb{Z})}^2;
\end{align}
the key steps in this reduction are that $\{ A_N*f^{\geq n} : N \}$ are all supported in $X$, and $\{ A_N*B_{m}^n, m \geq n, N \}$ are all supported in $E$. So, it suffices to prove that
\begin{align}
\sum_N \| A_N*\sum_{s \leq n} B_{n-s}^n \|_{\ell^2(\mathbb{Z})}^2 \lesssim \| f\|_{\ell^1(\mathbb{Z})}.
\end{align}
Expanding out the square, we compute
\begin{align}
    &\| A_N*\sum_{s \leq n} B_{n-s}^n \|_{\ell^2(\mathbb{Z})}^2 = \langle A_N*\tilde{A_N} * \sum_{s \leq n} B_{n-s}^n,\sum_{t \leq n} B_{n-t}^n  \rangle \\
    & \qquad = D_N^{-1} \| \sum_{s \leq n} B_{n-s}^n \|_{\ell^2(\mathbb{Z})}^2 + \langle \rho_N* \sum_{s \leq n} B_{n-s}^n, \sum_{t \leq n} B_{n-t}^n \rangle \\
    & \qquad \qquad + \sum_{|P| = |Q| = N, \ \text{dist}(P,Q) \leq N} O(N^{-\epsilon - 1} \cdot |P \cap E| \cdot |Q \cap E|) \\
    & \qquad \lesssim N^{\alpha -1 } \| f \cdot \mathbf{1}_{|f| \leq N^{1-\alpha}}\|_{\ell^2(\mathbb{Z})}^2 + |\langle \rho_N * \sum_{s \leq n} B_{n-s}^n, \sum_{t \leq s} B_{n-t}^n \rangle| + N^{-\epsilon} |E|,
\end{align}
see \eqref{e:set}.
Since the first and third term sum over $N \in 2^{\mathbb{N}}$ to $O(\| f \|_{\ell^1(\mathbb{Z})})$, see \eqref{e:1,1}, we only focus on the contribution of the second term. To this end, for each $t \leq s$, we will bound, for some $\kappa > 0$
\begin{align}\label{e:est}
    |\langle \rho_N *B_{n-s}^n, B_{n-t}^n \rangle| \lesssim 2^{-\kappa s} \| B_{n-t}^n \|_{\ell^1(\mathbb{Z})},
\end{align}
at which point we may sum
\begin{align}
    \sum_{N} |\langle \rho_N *\sum_{s \leq n} B_{n-s}^n, \sum_{t \leq s} B_{n-t}^n \rangle| &\lesssim
        \sum_{0 \leq t \leq s \leq n} |\langle \rho_N *B_{n-s}^n,  B_{n-t}^n \rangle| \\
    & \qquad \lesssim \sum_{t \leq n} 2^{- \kappa t} \| B_{n-t}^n \|_{\ell^1(\mathbb{Z})}
\end{align}
and so
\begin{align}
\sum_N |\langle \rho_N*\sum_{s \leq n} B_{n-s}^n,\sum_{t \leq n} B_{n-t}^n \rangle| &\lesssim \sum_{t} 2^{-\kappa t} \sum_{t \leq n} \| B_{n-t}^n \|_{\ell^1(\mathbb{Z})}  \\
& \qquad \lesssim \sum_t 2^{-\kappa t} \| f \|_{\ell^1(\mathbb{Z})} \lesssim \| f\|_{\ell^1(\mathbb{Z})}.
\end{align}
We turn to \eqref{e:est}, where it suffices to establish the pointwise bound
\begin{align}
    \| \rho_N * B_{n-s}^n \|_{\ell^{\infty}(\mathbb{Z})} \lesssim 2^{-\kappa s}.
\end{align}
Since $\text{supp } \rho_N \subset [-N,N]$, by translation invariance we may assume that $B_{n-s}^n$ is supported in $[0,N]$. 

Since $B_{n-s}^n$ has mean zero over dyadic intervals of length $2^{n-s}$, we can express
\begin{align}
    \sum_k \rho_N(x-k) B_{n-s}^n(k) &= \sum_{|Q| = 2^{n-s}} \big( \sum_{k \in Q} \rho_N(x-k) B_{n-s}^n(k) \big) \\
    &  \qquad = \sum_{|Q| = 2^{n-s}} \sum_{k \in Q} \big( \rho_N(x-k) - \rho_N(x-c_Q) \big) B_{n-s}^n(k),
\end{align}
where $x \in [-N,N]$, and $c_Q$ is (say) the left endpoint of each $Q$. So, regarding $x$ as arbitrary but fixed, we estimate the local contribution without exploiting any moment condition on the $\{ B_{n-s}^n\}$, simply using \eqref{e:set} to bound
\begin{align}
\sum_{|Q| = 2^{n-s}, \ \text{dist}(Q,x) \leq 2^{n-\kappa s}} \big( \sum_{k \in Q} \rho_N(x-k) B_{n-s}^n(k) \big) \lesssim N^{-1} \| B_{n-s}^n \mathbf{1}_{x + O(2^{n-\kappa s})} \|_{\ell^1(\mathbb{Z})} \lesssim 2^{-\kappa s},
\end{align}
where (possibly after decreasing $\kappa$), we may ensure that 
\[ N^{1-\alpha} = 2^{n(1-\alpha)} \ll 2^{n- s \kappa},\]
as $s \leq n$. And, in the complementary regime, whenever $|Q| = 2^{n-s}$ is such that 
\[ \text{dist}(Q,x) \geq 2^{n -\kappa s} \gg N^{1-\alpha},\]
we may bound
\[ |\sum_{k \in Q} \big( \rho_N(x-k) - \rho_N(x-c_Q) \big) B_{n-s}^n(k)| \lesssim \frac{|Q|}{N^2} \| B_{n-s}^n \|_{\ell^1(Q)},\]
so that
\begin{align}
    \sum_{|Q| = 2^{n-s}, \  \text{dist}(Q,x) \geq 2^{n-\kappa s}} \big( \sum_{k \in Q} \rho_N(x-k) B_{n-s}^n(k) \big) = O(2^{-s-n} \| B_{n-s}^n \|_{\ell^1([0,N])}) = O(2^{-s}).
\end{align}

\end{proof}

With this proof in hand, in the following section, we will show that the operators
\[  \frac{1}{N} \sum_{n \leq N} \delta_{ a_n } \]
(almost surely) satisfy the conditions of Proposition \ref{p:weak} whenever $\{ a_n \}$ are as in the statement of Theorem \ref{t:maindet} or \ref{t:mainran}.

\section{Examples}
We now show that appropriate averaging operators deriving from our sequences $\{ a_n \}$ satisfy the three conditions, \textbf{$\ell^2$-Boundedness; Sparse Support; and Reflection Symmetry}, with the third being the significant point. By convexity, see \cite[Page 23]{KMT}, Theorem \ref{t:maindet} or \ref{t:mainran} will follow from an analogous formulation involving the upper-half averages,
\[ \frac{1}{N} \sum_{N/2 < n \leq N} T^{a_n} f:\]
if we set
\[ \mathbb{D}_L := \mathbb{D}\cap[1,L], \]
and
\begin{align}
    a_N := \frac{1}{N} \sum_{n \leq N} T^{a_n} f, \; \; \; \tilde{a}_N := \frac{1}{N} \sum_{N/2 < n \leq N} T^{a_n} f
\end{align}
then by \cite[Page 23]{KMT}, and the second minimization in \eqref{e:jump}, for any $L$ we may dominate
\begin{align}
   \mathbf{N}(a_N : N \in \mathbb{D}_L ) &\leq \sum_{k =0}^L k^2 2^{-k} \cdot \mathbf{N}_k(\tilde{a}_{N/2^k} \mathbf{1}_{2^k \leq N } : N \in \mathbb{D}_L)  \\
   & \qquad + O(\sum_{k \leq L} \sum_{N \in \mathbb{D}_L} \frac{1}{N} \mathbf{1}_{2^k \leq N} |\tilde{a}_{N/2^k}|),
\end{align}
from which the result follows from a monotone convergence argument. In what follows, we will strict our averaging operators accordingly. 

\subsection{Deterministic Examples}
In this section, we define
\begin{align}\label{e:Aphi} A_N^{\phi} := \frac{1}{\varphi(N)} \sum_{(N/2)^{1/c} \leq n \leq N^{1/c}} \phi(n) \delta_{\lfloor n^c \rfloor},\end{align}
where $\phi$ is as in \eqref{e:phi}; we suppress the super-script when $\phi$ is constant, and set
\[ \varphi(N) := N^{1/c},\]
where we relate
\[ 1 < c := \frac{1}{1-\alpha} < 2,\]
so that we have $a_n = \lfloor n^c \rfloor$. Let 
\[
\mathbb{N}_c= \{\lfloor n^c \rfloor: n \in \mathbb{N}\},
\]
so that 
\[
|\mathbb{N}_c \cap[N]|=|\{\lfloor n^c \rfloor \leq N: n \in \mathbb{N} \}|= \lfloor N^{\frac{1}{c}} \rfloor = N^{\frac{1}{c}} + O(1).
\]
With $\phi_\alpha$ as in \eqref{e:phia}, set
\begin{equation}\label{e:detANBN}
B_N :=\frac{1}{c \varphi(N)}\sum_{N/2 < n \leq N} n^{-\alpha} \delta_{n} = \frac{1}{c N} \sum_{N/2 < n \leq N} \phi_\alpha(n/N) \delta_n;
\end{equation}
consolidate
\[
\mathcal{E}_N = A_N-B_N.
\]
The first elementary lemma concerns regularity properties of
\[ B_N^{\phi}:= \frac{1}{N} \sum_{N/2 < n \leq N} \phi(n/N) \delta_n.\]
\begin{lemma}\label{e:Bsmooth}
    For any $\phi$ as in \eqref{e:phi}
    \[ B_N^{\phi}*\tilde{B_N^\phi} = O(1/N),\]
and
\[ |B_N^{\phi}*\tilde{B_N^\phi}(x) - B_N^{\phi}*\tilde{B_N^\phi}(x+h)| \lesssim \frac{|h|}{N^2}.\]
\end{lemma}
\begin{proof}
    The first point is just convexity; for the second, we compute the discrete derivative
\begin{align}
    B_N^{\phi}*\tilde{B_N^\phi}(x) &= \frac{1}{N^2} \sum_{N/2< n,x+n \leq N} \phi((x+n)/N) \phi(n/N) \\
    & \qquad = B_N^{\phi}*\tilde{B_N^\phi}(x+1) + O(1/N^2)
\end{align}
by the regularity of $\phi$.   
\end{proof}

We first claim that, for any $1<c<2$, there exists $\epsilon = \epsilon(c) > 0$ so that
\begin{equation}\label{e:Ehat}
\|\widehat{\mathcal{E}_N}\|_{L^{\infty}(\mathbb{T})} \lesssim N^{- \epsilon};   
\end{equation}
note that this immediately implies quantitative convergence of the pertaining ergodic averages on $L^p(X)$, for any measure-preserving system, by Lemmas \ref{l:tri} and \ref{l:trans}.

The following Lemma is essentially given in \cite[{Proposition 2.1}]{CCH} as the regularity of our amplitudes can be used to reduce to the case of constant weights, and will be used to prove \eqref{e:Ehat}.

\begin{lemma}\label{l:spdexpestimate}
Suppose that $\alpha$ as stated above, $N$ is a sufficiently large integer, and $\phi$ is as in \eqref{e:phi}. Then for any $\theta \in \mathbb{T}$, and any $N/2 < t \leq N$, we have, for any $\epsilon>0$,
\[
\sum_{n \in \mathbb{N}_c\cap(N/2,t]}c \phi(n/N) n^{\alpha} e(n \theta) = \sum _{n \in (N/2,t]} \phi(n/N) e(n \theta) + \mathcal{E}^{\phi}_N(\theta;t) 
\]
where $\mathcal{E}^{\phi}_N(\theta;t) = O(N^{\frac{6}{5} - \frac{2}{5c} + \epsilon})$ uniformly in $\theta,t, \phi$.
\end{lemma}

Let us prove how to apply the above lemma to justify the stated upper bound for $\widehat{
\mathcal{E}}_N(\beta)$.

\begin{proof}[Verification of \eqref{e:Ehat}]
By partial summation 
\begin{align*}
\sum_{(N/2)^{1/c} < n \leq N^{\frac{1}{c}}} e(\lfloor n^c \rfloor \beta) & = \sum_{n \in \mathbb{N}_c \cap (N/2,N]}c n^\alpha e(n \beta) \frac{1}{c n^\alpha} \\ 
& = \int_{N/2}^{N} \frac{1}{ct^{\alpha}} \ d (\sum_{n \in \mathbb{N}_c \cap (N/2,t]}c n^\alpha e(n \beta))\\
& = \int_{N/2}^{N} \frac{1}{ct^{\alpha}} \ d (\sum_{n \in (N/2,t]} e(n \beta)) + \int_{N/2}^N \frac{1}{ct^{\alpha}} \ d \mathcal{E}_N(\theta;t) \\
& = \frac{1}{N^\alpha} \sum_{N/2 < n \leq N} \phi_\alpha(n/N) e(n \beta) +O(N^{1/2})
\end{align*}
by integration by parts.
\end{proof}

With the above in mind, by Lemma \ref{l:tri} we need only establish \textbf{Reflection Symmetry}, namely Property $(3)$ from our list of properties, as the second property is trivial; for the remainder of this section we only need to focus on decomposing 
\[
A_N*\tilde{A}_{N}(x).
\]

We begin by recalling the crucial van der Corput Lemma on exponential sums \cite[Satz 4]{VdC}, which will be used repeatedly to bound error terms which appear in our decomposition of $A_N*\tilde{A}_{N}(x)$.

\begin{lemma}[Van der Corput's Lemma]\label{Vdcptbound}
Assume that $a,b \in \mathbb{R}$ and $a<b$. Let $F \in C^2([a,b])$ be a real-valued function and $I$ be a subinterval of $[a,b]$. If there exists $\lambda >0$ and $v \geq 1$ such that 
\[
\lambda \lesssim |F''(x)| \lesssim v \lambda
\]
for every $x \in I \subset [a,b]$, where $I$ is a sub-interval, then we have 
$$
|\sum_{k \in I}e(F(k))| \lesssim v|I| \lambda^{1/2} + \lambda^{-1/2}.
$$
\end{lemma}

The following consequence is the key analytic input needed to establish our desired decomposition.
\begin{lemma}\label{l:expsumh1h2u1u2}
Let $N$ be a sufficiently large integer, $\theta \in [0,1)$, $1 \leq x \lesssim N$, $N/2 < t \leq N$, and $\phi$ be as in \eqref{e:phi}.
    \begin{enumerate}
    \item For any $u \in [0,1]$ and $1\leq |h| \leq N$, we have 
    \[
    \sum_{N/2 < n \leq t} \phi(n/N) e(n \theta - h(n+u)^{1/c})\lesssim N^{\frac{1}{2c}}|h|^{\frac{1}{2}} + N^{1-\frac{1}{2c}}|h|^{-\frac{1}{2}}.
    \]
    \item For any $u_1, u_2 \in [0,1]$, $1\leq |h_2| \leq |h_1| \leq N$ and $1<N_0 \leq N$, we have 
    \begin{align*}
     & \sum_{N/2<n \leq t} \phi(n/N) e(n\theta + h_1(n+u_1)^{\frac{1}{c}} + h_2(n+x+u_2)^{\frac{1}{c}})  \\
     & \lesssim N_0 + |h_1| |h_2|^{-\frac{1}{2}} \cdot N^{1+\frac{1}{2c}} \cdot |x|^{-\frac{1}{2}} \cdot N_0^{-\frac{1}{2}} + |h_2|^{-\frac{1}{2}} \cdot N^{2-\frac{1}{2c}} \cdot |x|^{-\frac{1}{2}} \cdot N_0^{-\frac{1}{2}} .
    \end{align*}
    \item Let $N^{2-\frac{2}{c}}<H \leq N$. For any $u_1, u_2 \in [0,1]$, if  $1 \lesssim (1+ \frac{100 |x|}{N})|h_2| \leq |h_1| \lesssim H $, then we have 
    \[
    \sum_{N/2<n \leq t} \phi(n/N) e(n\theta + h_1(n+u_1)^{\frac{1}{c}} + h_2(n+x+u_2)^{\frac{1}{c}})  \lesssim H^{\frac{1}{2}}N^{\frac{1}{2c}} .
    \]
    \end{enumerate}
\end{lemma}

\begin{proof}
By the regularity of $\phi$, we may reduce to the case of constant amplitude by summation by parts; we focus on the unweighted case in what follows:

Since the first point follows from Lemma \ref{Vdcptbound} directly, we turn to point $(2)$.

Let 
\[g(n)=n \theta + h_1(n+u_1)^{\frac{1}{c}} + h_2(n+x+u_2)^{\frac{1}{c}}.
\]
By a simple calculation, we have 
\[
g''(n) = c^{-1}(c^{-1} - 1) (h_1(n+u_1)^{\frac{1}{c}-2} + h_2(n+x+u_2)^{\frac{1}{c}-2}),
\]
so
\[
|g''(n)| \lesssim |h_1|N^{\frac{1}{c}-2}.
\]
In order to apply Lemma \ref{Vdcptbound}, we need to give a lower bound for $|g''(n)|$.
Let 
\[
g_1(n)=1 + \frac{h_2}{h_1}(1+\frac{x}{n})^{
\frac{1}{c}-2},\quad \text{and define $n_0$ via} \quad |g_1(n_0)|:=\min_{N<n \leq 2N}|g_1(n)|.
\]
Thus,
\[
|g''(n)| \approx |h_1n^{\frac{1}{c}-2}g_1(n)|.
\]

First, since $n_0, n \in (N,2N]$, we have
\[
|g_1(n) - g_1(n_0)| \gtrsim \frac{|h_2|}{|h_1|}|x| \frac{|n-n_0|}{N^2} 
\]
by the Mean-Value Theorem; below, we will only be interested in the case where
\[ |n-n_0| \geq N_0\]
and will just trivially bound
\begin{align}
    |\sum_{|n-n_0| \lesssim N_0} e(g(n))| \lesssim N_0.
\end{align}

In particular, when 
\[ |g_1(n_0)| \ll \frac{|h_2|}{|h_1|}|x| \frac{N_0}{N^2},\] this implies
\begin{align}
    |g_1(n)| \gtrsim \frac{|h_2|}{|h_1|}|x| \frac{|n-n_0|}{N^2},
\end{align}
and thus
\begin{align}
    |g''(n)| \gtrsim |h_2| |x| N^{1/c-4} |n-n_0|.
\end{align}
On the other hand, if 
\[ |g_1(n_0)| \gtrsim \frac{|h_2|}{|h_1|}|x| \frac{N_0}{N^2}\]
then
\begin{align}
    |g''(n)| \gtrsim |h_2| |x| N_0 N^{1/c-4}.
\end{align}

The upshot is that whenever $|n-n_0| \gg N_0$, we may bound
\begin{align}\label{e:approxphase}
    |h_2| |x| N_0 N^{1/c-4} \lesssim |g''(n)| \lesssim |h_1| N^{1/c-2} = 
    \big( |h_2| |x| N_0 N^{1/c-4} \big) \cdot \frac{|h_1| N^2 }{|h_2 x| N_0}.
\end{align}
So, point $(2)$ above follows directly from van der Corput's Lemma \ref{Vdcptbound}.
We focus on the third point. Write
\[
g(n)= n \theta + (h_1+h_2)(n+u_1)^{\frac{1}{c}} + h_2\big( (n+x+u_2)^{\frac{1}{c}}-(n+u_1)^{\frac{1}{c}} \big).
\]
By the triangle inequality $|h_1+h_2| \geq |h_1|-|h_2| \geq \frac{100|x|}{N}|h_2|$, and by the mean-value theorem,
\[
|h_2[(n+x+u_2)^{\frac{1}{c}}-(n+u_1)^{\frac{1}{c}}]| \leq 5|h_2| x n^{\frac{1}{c}-1}\leq \frac{1}{10}|(h_1+h_2)(n+u_1)^{\frac{1}{c}}|.
\]
Therefore, 
\[ |g(n) - \big(  n \theta + (h_1+h_2)(n+u_1)^{\frac{1}{c}} \big)| \leq \frac{1}{5} | (h_1+h_2)(n+u_1)^{\frac{1}{c}}|,\]
and 
\[
g''(n) \approx ((h_1+h_2)(n+u_1)^{\frac{1}{c}-2}.
\]
We apply Lemma \ref{Vdcptbound} to get
\begin{align*}
|\sum_{N/2<n \leq t} e(n\theta + h_1(n+u_1)^{\frac{1}{c}} + h_2(n+x+u_2)^{\frac{1}{c}})|  \lesssim H^{\frac{1}{2}}N^{\frac{1}{2c}}.
\end{align*}

\end{proof}

The following is well known, see e.g.\ \cite[\S 2]{HB}.
\begin{lemma}\label{fracPrt}
Let $\psi(t) = \{x\} - \frac{1}{2}$. For each $H \geq 2$ and $t \in \mathbb{R}$, we have
\begin{align}
\psi(t) & = - \frac{1}{2 \pi i} \sum_{|h|=1}^{\infty} \frac{e(ht)}{h}\\
& =-\frac{1}{2 \pi i} \sum_{|h| \leq H} \frac{e(ht)}{h} + O(\min\{1, \frac{1}{H \|t\|}\}).
\end{align}
Furthermore,
\begin{equation}\label{Errfrc}
\min\{1, \frac{1}{H\|t\|}\}= \sum_{h= -\infty}^{+\infty}b_{h} e(ht),
\end{equation}
where 
\[
b_h \lesssim \min \{\frac{\log H}{H}, \frac{H}{h^2}\}.
\]
\end{lemma}
Combining Lemma \ref{fracPrt} with Lemma \ref{l:expsumh1h2u1u2}, we will obtain the following lemma.

\begin{lemma}\label{l:min1Hn1c}
Let $N$ be a sufficiently large integer and set $H := N^{2-2/c + \delta}$ with any sufficiently small $\delta=\delta(c)>0$. Then for any $1<c<6/5$ and $0 \leq u < 1$,
\[
\sum_{N< n \leq 2N} \min\{1, \frac{1}{H \|(n+u)^{\frac{1}{c}}\|}\} \lesssim N^{\frac{2}{c}-1-\frac{\delta}{2}}.
\]
\end{lemma}

\begin{proof}
By applying Lemma \ref{fracPrt} and Lemma \ref{l:expsumh1h2u1u2} (1), 
\begin{align}\label{N11c1Hn1c}
&\sum_{N< n \leq 2N}\min\bigg\{1,\frac{1}{H\|(n+u)^{\frac1c}\|}\bigg\}\notag\\
=& \sum_{h=-\infty}^\infty b_h\sum_{N< n \leq 2N} e(h(n+u)^{\frac1c})\lesssim
\frac{N \log H }{H} + \sum_{|h| \geq 1} |b_h|\cdot \sum_{N< n \leq 2N} e(h(n+u)^{\frac1c})\\
= & \frac{N}{H}\log H + \sum_{1 \leq |h| \leq N^{1-\frac{1}{c}}}\frac{\log H}{H} | \sum_{N< n \leq 2N} e(h(n+u)^{\frac1c}) | \\
& \qquad + \sum_{|h| > N^{1-\frac{1}{c}}}\min\left\{\frac{\log H}{H},\frac{H}{h^2} \right\} |\sum_{N< n \leq 2N} e(h(n+u)^{\frac1c}) | \\
\lesssim & N^{\frac{2}{c}-1-\frac{\delta}{2}} + N^{\frac{1}{c}-\frac{1}{2} - \frac{\delta}{2}} + N^{1-\frac{1}{2c} + \frac{3\delta}{4}} \lesssim N^{\frac{2}{c}-1-\frac{\delta}{2}}.
\end{align}
Therefore, the claim follows from $1<c<6/5$.
\end{proof}

It is time to estimate $A_N*\tilde{A}_{N}(x)$; the Fourier inversion identity
\begin{equation}\label{e:Fcovn^c}
A_{N}*\tilde{A}_{N}(x) = \int_{0}^{1} \widehat{A_N}(\beta)\widehat{A_N}(-\beta)e(\beta x) d \beta 
\end{equation}
will be used without comment.

Beginning with the expansion of the indicator function
\begin{align}\label{e:sawtooth}
\mathbf{1}_{n \in \mathbb{N}_c} &= \lfloor-n^{\frac{1}{c}} \rfloor-\lfloor-(n+1)^{\frac{1}{c}}\rfloor \\
& \qquad = \big( (n+1)^{1/c} - n^{1/c} \big) + \big( \psi(-(n+1)^{1/c}) - \psi(-n^{1/c}) \big), 
\end{align}
we use Lemma \ref{fracPrt} to express
\begin{align}
\notag & \varphi(N)\widehat{A_N}(\beta)\\ \notag
& = \sum_{(N/2)^{1/c} <n \leq N^{1/c}} e(\lfloor n^c \rfloor \beta)  \\ \notag
& = \sum_{N/2 < m \leq N} e(m \beta)  (\lfloor - m^{\frac{1}{c}}\rfloor - \lfloor - (m+1)^{\frac{1}{c}}\rfloor)\\ \notag
& = \sum_{N/2 < m \leq N} e(m \beta)  (\psi(-(m+1)^{\frac{1}{c}}) - \psi(-m^{\frac{1}{c}})) + \sum_{N/2 < m \leq N} e(m \beta)  ((m+1)^{\frac{1}{c}} - m^{\frac{1}{c}})\\ \notag
& = \frac{1}{c N^{1-1/c}}\sum_{N/2 < m \leq N} e(m \beta) \phi_\alpha (m/N)  + 
\sum_{N/2 < m \leq N} e(m \beta) (\psi(-(m+1)^{\frac{1}{c}}) - \psi(-m^{\frac{1}{c}})) \\ \notag
& \qquad \qquad \qquad \qquad + \widehat{E_N}(\beta)\\ \notag
&  =  \frac{N^{1/c}}{c} \widehat{B_N}(\beta)  - \frac{1}{2\pi i} \sum_{N/2 < m \leq N} e(m \beta) \sum_{1 \leq |h| \leq H} \frac{1}{h}(e(-h(m+1)^{\frac{1}{c}}) - e(-hm^{\frac{1}{c}})) \\ \notag
& \qquad \qquad \qquad - \frac{1}{2\pi i} \sum_{N/2 < m \leq N} e(m \beta) \sum_{|h|> H} \frac{1}{h}(e(-h(m+1)^{\frac{1}{c}}) - e(-hm^{\frac{1}{c}})) + \widehat{E_N}(\beta) \\ 
&  =:\widehat{f_{N,s}}(\beta) +\widehat{f_{N, 1}}(\beta) + \widehat{f_{N, 2}}(\beta) + \widehat{E_{N}}(\beta), \label{e:FourierDecompose4cases}
\end{align}
where
\begin{align}\label{e:smallptwise}
|E_N| \lesssim N^{1/c-2} \mathbf{1}_{(N/2,N]}
\end{align}
is an error term, with the gain coming from the second order Taylor expansion of $t \mapsto (m+t)^{1/c}$. In particular, we have emphasize the decomposition
\begin{align}
   c n^{1-1/c} \mathbf{1}_{\mathbb{N}_c} = 1 + f_{N,1} + f_{N,2} + E_N.
\end{align}

By choosing $H = N^{2-2/c + \delta}$ with sufficiently small $\delta>0$ and applying Lemma \ref{l:expsumh1h2u1u2} (1) and Lemma \ref{l:min1Hn1c}, respectively, we have the following uniform bounds,
\[
\widehat{f_{N, 1}}(\beta) \lesssim N^{1-\frac{1}{2c} }, \; \; \; \widehat{f_{N,2}}(\beta) \lesssim N^{\frac{2}{c}-1-\delta}.
\]

By the above arguments, we may expect that the main term will come from ${B_N}$ and the rest will contribute errors. Before estimating the error terms of $A_N*\tilde{A}_{N}(x)$, let us prove the following facts:
\begin{lemma}\label{l:4pptfdet}
Let $f_{N,s}, f_{N,1}, f_{N,2}$ be as before and $N$ be sufficiently large, $1<c<6/5$ and $H=N^{2-2/c+\delta}$ with sufficiently small $\delta>0$. We have the following bounds:
\begin{enumerate}
    \item $
    \| \widehat{f_{N, s}} \|_{L^1(\mathbb{T})} \lesssim N^{\frac{1}{c}-1} \log N
    $;
    \item $ 
    |f_{N, 1}*\tilde{f}_{N,1}(x)| \lesssim N^{1-\frac{1}{3c}+\delta}$ whenever $N^{1/c} \lesssim |x| \leq N$; and
    \item  $f_{N,i}$ are supported on $(N/2,N]$ and satisfy the pointwise bound
    \[ |f_{N,1}(x)| \lesssim \log N, \; \; \; |f_{N,2}(x)| \lesssim \sum_{u \in \{0,1\}} \min\{1,\frac{1}{H\|(x+u)^{1/c}\|} \}.\]
\end{enumerate}
Consequently, 
\[ |f_{N,s}*\tilde{f}_{N,i}| \lesssim N^{2/c-1-\delta/3}  \qquad\text{and} \qquad |f_{N,s}*\tilde{E}_N| \lesssim N^{2/c-2} \]
and
for each $i = 1,2$
\[ |f_{N,2}*\tilde{f}_{N,i}| + |E_N*\tilde{f}_{N,i}| + |E_N*\tilde{E}_N| \lesssim   N^{2/c-1-\delta/3}. \]
\end{lemma}
\begin{proof}
The pointwise bounds are a consequence of direct computation and the Riemann-Lebesgue Lemma, to bound, for each $i = 1,2$, noting that $1-\frac{1}{2c}<\frac{2}{c}-1$,
\begin{align}
    \| f_{N,s} *f_{N,i} \|_{\ell^{\infty}(\mathbb{Z})} \leq \| \widehat{f_{N,s}} \widehat{f_{N,i}} \|_{L^1(\mathbb{T})}  \leq \| \widehat{f_{N,s}} \|_{L^1(\mathbb{T})} \| \widehat{f_{N,i}} \|_{L^{\infty}(\mathbb{T})}\\
    \lesssim   N^{2/c-1-\delta/3}.
\end{align}
Thus, we focus on the first three points; the first is a standard Dirichlet kernel argument and the third is straightforward, so it remains to address the first two.

Let
\[K_N(h_1,h_2;x) := \sum_{u,v \in \{0,1\}} \big| \sum_{N/2 < x+m,m \leq N} e(h_2(m+x+u)^{\frac{1}{c}} + h_1 (m+v)^{\frac{1}{c}}) \big| \]
To prove (2), we first note that  
\begin{equation}\label{e:f_1_convolve_f_1}
|f_{N, 1} * \tilde{f}_{N ,1}(x)| \lesssim \sum_{1 \leq |h_2| \leq |h_1| \leq H} \frac{1}{|h_1| |h_2|} K_N(h_1,h_2;x)
\end{equation}
In order to apply Lemma \ref{l:expsumh1h2u1u2} efficiently, we decompose the right-hand-side of \eqref{e:f_1_convolve_f_1} into different pieces according to the size of $|h_1|$ and $|h_2|$. The right-hand-side of \eqref{e:f_1_convolve_f_1} is bounded by 
\begin{align*}
& \left( \sum_{\substack{1 \leq |h_2| \leq |h_1| \leq  H \\ 1 \leq |h_2| \leq \frac{N}{100|x|}}} + \sum_{\frac{100N}{|x|} \leq |h_2| \leq |h_1| \leq  H } \right) \frac{1}{|h_1| |h_2|} K_N(h_1,h_2;x)\\
&= \left( \sum_{\substack{1 \leq |h_2| \leq |h_1| \leq  H \\ 1 \leq |h_2| \leq \frac{N}{100|x|}}} + \sum_{\substack{ \frac{N}{100|x|} \leq |h_2| \leq  H \\ 0 \leq |h_1|-|h_2| \leq \frac{100|x|}{N}|h_2|}} +  \sum_{\substack{ \frac{100N}{|x|} \leq |h_2| \leq  H \\ |h_1|-|h_2| > \frac{100|x|}{N}|h_2|}} \right) \frac{1}{|h_1| |h_2|} K_N(h_1,h_2;x)\\
&=: \Sigma_1 + \Sigma_2 + \Sigma_3
\end{align*}
We first bound $\Sigma_2$ and $\Sigma_3$ by using points $(2)$ and $(3)$ respectively from Lemma \ref{l:expsumh1h2u1u2}. With $N_0$ a free parameter, 
\[ H = N^{2-2/c+\delta}, \; \; \; N^{1/c} \leq |x| \leq N \]
we bound
\begin{align}
    \Sigma_2 \lesssim N_0 N^{\delta/10} + N^{3/2 - 1/2c + \delta} N_0^{-1/2};
\end{align}
indeed, we bound
\begin{align}\label{e:largexbad}
    \Sigma_2 &\lesssim N_0 \sum_{\substack{ \frac{N}{100|x|} \leq |h_2| \leq  H \\ 0 \leq |h_1|-|h_2| \leq \frac{100|x|}{N}|h_2|}} \frac{1}{|h_1h_2|} \\
    & + N^{1+1/2c} |x|^{-1/2} N_0^{-1/2} \sum_{\substack{ \frac{N}{100|x|} \leq |h_2| \leq  H \\ 0 \leq |h_1|-|h_2| \leq \frac{100|x|}{N}|h_2|}} \frac{1}{|h_2|^{3/2}} \\
    &  + N^{2-1/2c} |x|^{-1/2} N_0^{-1/2} \sum_{\substack{ \frac{N}{100|x|} \leq |h_2| \leq  H \\ 0 \leq |h_1|-|h_2| \leq \frac{100|x|}{N}|h_2|}} \frac{1}{|h_1| \cdot |h_2|^{3/2}} \\
    & \lesssim N_0 \log N + N^{1/2c} |x|^{1/2} N_0^{-1/2} H^{1/2} + N^{3/2-1/2c} N_0^{-1/2} \\
    & \lesssim N_0 \log N +  N^{3/2 - 1/2c + \delta/2} N_0^{-1/2},
\end{align}
so optimizing $N_0 = N^{1-1/3c}$ yields a bound
\begin{align}
    \Sigma_2 \lesssim N^{1-1/3c + \delta}.
\end{align}

And, we bound
\begin{align}
    \Sigma_3 \lesssim N^{1-1/2c + \delta},
\end{align}
as we may bound
\begin{align}
    K_N(h_1,h_2;x) \lesssim H^{1/2} N^{1/2c} = N^{1-1/2c+\delta/2}
\end{align}
inside the whole region by Lemma \ref{l:expsumh1h2u1u2}, and the summation introduces a factor of $(\log N)^2 \ll N^{\delta/2}$.

We now turn to bound $\Sigma_1$. If $h_1=-h_2$, then 
\[
K_N(h_1,h_2;x) = \sum_{u,v \in \{0,1\}} \big| \sum_{N/2 < x+m,m \leq N} e(h_2(m+x+u)^{\frac{1}{c}}- h_2 (m+v)^{\frac{1}{c}}) \big|.
\]
Since 
\[
|g''(m)|=|\Big( h_2 \big( (m+x)^{\frac{1}{c}}-  m^{\frac{1}{c}}\big) \Big)''| \approx |h_2| N^{\frac{1}{c}-3}|x|,
\]
by Lemma \ref{Vdcptbound},
\begin{align}
    K_N(-h_2,h_2;x) &\lesssim |h_2|^{1/2} N^{1/2c} + |h_2|^{-1/2} N^{3/2-1/c} \\
    & \lesssim N^{1-1/2c + \delta/2} + N^{3/2-1/c} \lesssim N^{3/2-1/c}
\end{align}
provided that $\delta = \delta(c) > 0$ is sufficiently small. Consequently,
\begin{align}
    \sum_{h_1 = -h_2, \ |h_2| \leq H} \frac{1}{|h_1h_2|} K_N(h_1,h_2;x) \lesssim N^{3/2-1/c}.
\end{align}

We now address the off-diagonal case; namely, we may bound $|h_1+h_2| \geq 1$.

By a simple computation, crucially using that $1 \leq |h_2| \leq \frac{N}{100 |x|}$ is so small and that we are in the off-diagonal regime, we have 
\[
|g''(m)|=|\big( h_1 (m+v)^{\frac{1}{c}}+h_2(m+x+u)^{\frac{1}{c}}\big)''| \approx |h_1 + h_2|N^{\frac{1}{c}-2}, 
\]
so by Lemma \ref{Vdcptbound}, 
\[
K_N(h_1,h_2;x) \lesssim N^{1-1/2c + \delta/2},
\]
so
\begin{align}
    \sum_{1 \leq |h_2| \leq |h_1| \leq H, \ 1 \leq |h_2| \leq N/100|x|} \mathbf{1}_{|h_1 + h_2| \geq 1} K_N(h_1,h_2;x) \lesssim N^{1-1/2c + \delta},
\end{align}
and thus
\[
\Sigma_1 \lesssim N^{3/2 -1/c},
\]
provided $\delta = \delta(c)> 0$ is sufficiently small.
\end{proof}

In particular, when $N^{1/c} \lesssim |x| \lesssim N$, we have established the following decomposition; the numerology from point (2) above is what determines our upper bound on $c$.

\begin{cor}
    Suppose that $N^{1/c} \lesssim |x| \lesssim N$ and $1 < c < 7/6$. Then there exists $\epsilon = \epsilon(c) > 0$ so that
    \[ A_N*\tilde{A_N}(x) = B_N*\tilde{B_N}(x) + O(N^{-\epsilon - 1}). \]
\end{cor}

By Lemma \ref{e:Bsmooth}, it remains only to address the upper bound
\begin{align}\label{e:upperbound}
A_N*\tilde{A_N}(x) = O(1/N), \; \; \; 0 < |x| \lesssim N^{1/c}
\end{align}
when $0 < |x| \lesssim N^{1/c}$.
This will be the content of our final calculation in this subsection, from which derives our global restriction $1 < c < 7/6$: we need to balance
\[ N^{1-1/3c + \delta} \lesssim N^{2/c-1} \Rightarrow 1 < c< 7/6 \text{ provided that } 1 \gg \delta = \delta(c) >0,\]
where the term on the left derives from our computation of $\Sigma_2$.

\begin{proof}[Proof of \eqref{e:upperbound}]
Without loss of generality, we consider the case of positive $x \geq 1$.

Suppose that $(N/2)^{1/c}< m<n \leq N^{1/c}$ are such that
\[ \lfloor n^c \rfloor - \lfloor m^c \rfloor = x;\]
then we must have
\[ x-1 \leq n^c - m^c \leq x + 1.\]
For notational ease, denote
\[ k := k(n,m) := n -m, \]
so that $ |k| \approx \frac{|x|}{N^{1-1/c}}$ by the Mean-Value Theorem.

Consider the increasing function 
\[ g_k(n) := n^c - (n-k)^c,\]
and note that by the Mean-Value Theorem
\[ |g_k(n+1) - g_k(n)| \gtrsim \frac{|k|}{N^{2/c-1}}.\]
So
\begin{align}
    \varphi(N)^2 A_N*\tilde{A_N}(x) &\leq |\{ (n,k) : |k| \approx \frac{|x|}{N^{1-1/c}}, \ |n^c - (n-k)^c -x| \leq 1 \}| \\
    & \qquad \lesssim \sum_{|k| \approx \frac{|x|}{N^{1-1/c}}} \frac{N^{2/c-1}}{k} + 1 \lesssim N^{2/c-1},
\end{align}
from which the result follows.
\end{proof}

This concludes our study of our deterministic sequences
\[ a_n = \lfloor n^c \rfloor, 1 < c < 7/6.\]
In the next subsection, we treat our random examples, which are significantly less involved.

\subsection{Random Examples}
Let $X_n$ denote a sequence of independent Bernoulli Random Variables with expectations $\mathbb{E} X_n = n^{-\alpha},\ 0 < \alpha < 1/2$, and define the hitting times
\[ a_k := \min\{ n : X_1 + \dots + X_n = k \},\]
so that (almost surely) $a_k \approx k^{\frac{1}{1-\alpha}}$ by Chernoff's inequality, Lemma \ref{l:Chern} below, and a Borel-Cantelli argument, see \cite[\S 5]{BOOK}. 

\begin{lemma}\label{l:Chern}
Let $\{ Y_n \}$ denote independent mean-zero $1$-bounded random variables. Then
\[ \mathbb{P}( |\sum_{n \leq N} Y_n| \geq \lambda) \leq 10\max\{ e^{- \frac{\lambda^2}{10 V_N}}, e^{-\frac{\lambda}{10}} \}\]
where $V_N := \sum_{n \leq N} \mathbb{E} |Y_n|^2$ is the total variance.
\end{lemma}

For all $N$ sufficiently large, set 
\[ A_N^0 := \frac{1}{\sum_{N/2 < n \leq N} X_n} \sum_{N/2 < n \leq N} X_n \delta_n, \]
which is a reparametrization of
\[ \frac{1}{N} \sum_{N/2 < n \leq N} \delta_{a_n},\]
and with
\[ W_N := \sum_{N/2 < n \leq N} \mathbb{E} X_n \approx_\alpha N^{1-\alpha}, \]
so $W_N \approx \varphi(N)$ in the previous notation,
define the random averaging operator
\[ A_N := \frac{1}{W_N} \sum_{N/2 < n \leq N} X_n \delta_{n},\]
its deterministic counterpart
\[ 
B_N := \mathbb{E} A_N := \frac{1}{W_N} \sum_{N/2 <  n \leq N} n^{-\alpha} \delta_{n} \approx_\alpha \frac{1}{N} \sum_{N/2 < n \leq N} \phi_\alpha(n/N) \delta_n,\]
and consolidate the error term
\[ \mathcal{E}_N := A_N - B_N.\]
First, note that, almost surely
\begin{align}
    \| A_N - A_N^0 \|_{\ell^1(\mathbb{Z})} \lesssim \frac{|\sum_{n \leq N} (X_n - \mathbb{E} X_n)|}{ W_N} \lesssim    \sqrt{\log N} \cdot N^{- (1-\alpha)/2}
\end{align}
by Chernoff's inequality, since almost surely
\begin{align}\label{e:conc} \sum_{N/2 < n \leq N} X_n = \frac{1}{1-\alpha} N^{1-\alpha} + O(\sqrt{\log N} \cdot N^{(1- \alpha)/2}), \end{align}
so
\begin{align}
    \| \mathbf{N}((A_N - A_N^0)*f) \|_{\ell^{1}(\mathbb{Z})} \lesssim \sum_N \|(A_N - A_N^0)*f\|_{\ell^1(\mathbb{Z})} \lesssim \| f \|_{\ell^1(\mathbb{Z})}; 
\end{align} 
thus, we can focus our attention on $A_N$. First, note that
\[ A_N*\tilde{A_N}(0) \approx_\alpha N^{\alpha - 1}, \]
by \eqref{e:conc}. Next, by \cite[\S 5]{BOOK}, with probability $1$, we have
\[
\|\widehat{\mathcal{E}}_{N}\|_{L^{\infty}(\mathbb{T})} \lesssim \sqrt{\log N} \cdot N^{(\alpha -1)/2},
\]
and
\[ \| \mathcal{E}_N * \tilde{\mathcal{E}_N} \|_{\ell^{\infty}(\mathbb{Z} \smallsetminus \{0\})} \lesssim N^{-1 -\epsilon(\alpha)} \]
 for some $\epsilon(\alpha) > 0$. Moreover,
\[ B_N*\tilde{\mathcal{E}_N}(x) = \frac{1}{W_N^2} \sum_{N/2 < n,x+n \leq N} (X_n - \mathbb{E} X_n) \cdot (x+n)^{-\alpha} \]
and so by Chernoff's inequality, almost surely
\[ \| B_N*\tilde{\mathcal{E}_N} \|_{\ell^{\infty}(\mathbb{Z})} + \| \mathcal{E}_N*\tilde{B_N} \|_{\ell^{\infty}(\mathbb{Z})} \lesssim \log N  \cdot  N^{\alpha - 2} \ll N^{-3/2}.\]

Finally, the contribution of $B_N*\tilde{B_N}$ has been handled by Lemma \ref{e:Bsmooth}.

\section{Pointwise Convergence}
In this section, we apply our main results to prove quantitative convergence estimates for our ergodic averages,
\[ \mathbb{E}_{[N]} T^{a_n} f \]
for $f \in L^1(X)$.

We first emphasize that from the purposes of convergence, it suffices to establish convergence along lacunary sequences of the form
\[ \{ N = \lfloor 2^{k/R} \rfloor : k \geq 1 \}_{R \in 2^{\mathbb{N}}}, \]
so we will restrict all times below to such a sequence; we will regard this sequence as fixed throughout the below discussion.

We begin by exploring the utility of our operators \eqref{e:osc} in questions involving pointwise convergence: observe that a norm estimate of the form e.g.\
\[  \sup_{\epsilon} \| \epsilon N_{\epsilon}(f_n)^{1/r} \|_{L^{p,\infty}(X,\mu)} \leq C\]
implies that $N_{\epsilon}(f_n) <  \infty$ $\mu$-almost everywhere for each $\epsilon > 0$, and thus $\{ f_n\}$ converge almost everywhere as well; via the majorization
\[ \epsilon N_{\epsilon}(f_n)^{1/r} \leq \mathcal{V}^r(f_n),\]
we see similar utility in working with $r$-variation. Or, more subtly, a norm estimate of the form
\[ \sup_{\{ M_j \}} \| \mathcal{O}_{\{ M_j \}_{j \leq J}}(f_n) \|_{L^2(X)} = o_{J \to \infty}(J^{1/2}) \]
implies that for each $\epsilon > 0$
\[ \mu (\{ \limsup_{n,m} |f_n - f_m| \geq \epsilon \}) = 0, \]
as otherwise one could extract a (finite) increasing sequence of times, $\{ M_j \}_{j \leq J}$, of arbitrary length, so that
\[ \mu(\{ \max_{M_j \leq  n \leq M_{j+1}} |f_n - f_{M_j}| \geq \epsilon/10 \}) > \epsilon_0
\]
for some $\epsilon_0 > 0$, leading to the following chain of inequalities
\[ J \epsilon_0 \leq \sum_{j \leq J} \mu(\{ \max_{M_j  \leq n \leq M_{j+1}} |f_n - f_{M_j}| \geq \epsilon/10 \}) \lesssim \epsilon^{-2} \| \mathcal{O}_{\{M_j\}_{j \leq J}}(f_n)\|_{L^2(X)}^2;
\]
but this is precluded by the slow growth rate of $\mathcal{O}$. This approach was introduced, and crucially used, by Bourgain in \cite{B0, Bp, B2}. 

More generally, we have the following lemma, which we use to deduce Theorem \ref{t:cor} from Proposition \ref{p:weak}.
\begin{lemma}\label{l:trans}
Suppose that $f_N(x) := \mathbb{E}_{[N]} f(x-a_n)$, and that one of the following estimates hold (for any $r < \infty$)
\begin{itemize}
    \item $\sup_{\epsilon > 0} \| \epsilon N_\epsilon(f_N)^{1/r} \|_{\ell^{p,\infty}(\mathbb{Z})} \lesssim \|f \|_{\ell^p(\mathbb{Z})};$
    \item $\| \mathcal{V}^r(f_N) \|_{\ell^{p,\infty}(\mathbb{Z})} \lesssim \| f\|_{\ell^p(\mathbb{Z})};$ or
    \item $\sup_{\{M_j\}} \| \mathcal{O}_{\{M_j\}_{j \leq J}}(f_N) \|_{\ell^{p,\infty}(\mathbb{Z})} = o_{J \to \infty}(J^{1/2}) \| f \|_{\ell^p(\mathbb{Z})}.$ 
\end{itemize}
Then for each $f \in L^p(X)$, $\{ \mathbb{E}_{[N]} T^{a_n} f \}$ converge almost everywhere.
\end{lemma}
\begin{proof}
The first two alternatives follow from straightforward applications of Calder\'{o}n's Transference Principle \cite{C1}; the interesting case involves the oscillation operator.

First, by monotone convergence, we have a weak-type $(p,p)$ bound on 
    \[ \| \sup_N |f_N| \|_{\ell^{p,\infty}(\mathbb{Z})} \lesssim \|f \|_{\ell^p(\mathbb{Z})},\] and so by Calder\'{o}n's Transference Principle \cite{C1}, on $\sup_N |\mathbb{E}_{[N]} T^{a_n} f|$ as well, 
    \[ \| \sup_N |\mathbb{E}_{[N]} T^{a_n} f | \|_{L^{p,\infty}(X)} \lesssim \|f\|_{L^p(X)};\]
    it suffices to prove pointwise convergence for bounded functions. Assume for concreteness that $h$ is an increasing function such that $a_n \leq h(n)$ for all $n$.
        By \cite[Proposition 4.3]{BKNew} and \cite[Proposition 5.5]{BKNew}, it suffices to prove that
    \begin{align}
        C_{\tau,H} := \sup\{ K : &\text{ there exists $1$-bounded } f, M_0 < M_1 < \dots < M_K \leq h^{-1}(H/100), \\
        &\text{ so that } |\{ x \in [H] : \max_{M_{k-1} \leq N \leq M_k} |f_N - f_{M_k}| \geq \tau \}| \geq \tau H \} \lesssim_{\tau} 1
    \end{align} 
    independent of $H$, where we specialize all times to live in $\mathbb{D}_{\lambda} = \{ \lfloor 2^{k/R} \rfloor \} $ with $R \gg \tau^{-1}$.

In the interest of keeping the paper self-contained, we provide the details below, and continue with the proof:

So, suppose that $J := C_{\tau,H} \lesssim_H 1$ realizes the above supremum, for an appropriate $f, \{M_k\}$; since $M_J \leq h(H/100)$, we can assume that $|f| \leq \mathbf{1}_{[-H,2H]}$.
Then
\begin{align}
    \tau H &\leq \frac{1}{J} \sum_{j \leq J} \sum_{x \in [H]} \mathbf{1}_{\{ \max_{M_{j-1} \leq N \leq M_j} |f_N - f_{M_j}| \geq \tau \}}(x) \\
    &  \leq \tau^{-2} \sum_{x \in [H]} \frac{1}{J} \sum_{j \le J} \max_{M_{j-1} \leq N \leq M_j} |f_N - f_{M_j}|^2(x) \\
    &\leq \tau^{-2} \tau^{100} H  + \tau^{-2} |\{\big(\frac{1}{J} \sum_{j \leq J} \max_{M_{j-1} \leq N \leq M_j} |f_N  - f_{M_j}|^2\big)^{1/2} \geq \tau^{10}\}| \\
    & \leq \tau^{10} H + o_{J \to \infty}(\tau^{-2 -10p} \| f \|_{\ell^p(\mathbb{Z})}^p) = \tau^{10} H + o_{J \to \infty}(\tau^{-2 -10p} H)
\end{align} 
    which forces an upper bound on $C_{\tau,H}$, independent of $H$.
\end{proof}

We now complete the reductions that allow us to derive convergence from the boundedness of $C_{\tau,H} \lesssim_\tau 1.$
\begin{proposition}\label{p:pointwise}
    Suppose that for each $\tau,H$, $C_{\tau,H} \lesssim_{\tau} 1$
    independent of $H$. Then for any $(X,\mu,T)$, and any $f \in L^{\infty}(X)$,
    \[ \mathbb{E}_{[N]} T^{a_n} f \]
converges almost everywhere.
    \end{proposition}
The proof of Proposition \ref{p:pointwise} goes through the Calder\'{o}n Transference principle, with the principle quantity of interest being a measure-theoretic variant of $C_{\tau,H}:$
\begin{align}
C_{\tau,H}^{(X,\mu,T)} &:= \sup \Big\{ K : \text{ there exists a $1$-bounded } f , \\
& \qquad \text{ and a sequence } M_0 < M_1 < \dots < M_K \leq h^{-1}(H/100), \\
& \qquad \qquad \text{ so that } \mu(\{ \max_{M_{k-1} \leq M \leq M_k} |\mathbb{E}_{[M]} T^{a_n} f - \mathbb{E}_{[M_k]} T^{a_n} f| \geq \tau \}) \geq \tau \Big\}
\end{align} 
where we restrict all times to live in the set $\{ \lfloor 2^{k/R} \rfloor \} $ where $R \gg \tau^{-1}$;
convergence in $L^{\infty}(X,\mu,T)$ follows from a bound $C_{\tau,H}^{(X,\mu,T)} \lesssim_{\tau} 1$: 
\begin{lemma}[Quantifying Convergence]\label{l:QC}
    To prove that $\mathbb{E}_{[M]}T^{a_n}f$ converge almost everywhere for each $f \in L^{\infty}(X,\mu,T)$, it suffices to prove that for each $\tau,H$, 
    \[ C_{\tau,H}^{(X,\mu,T)} \lesssim_{\tau} 1, \]
uniformly in $H$.
\end{lemma}
\begin{proof}
Set
    \[ C_{\tau}^{(X,\mu,T)} := C_{\tau,\infty}^{(X,\mu,T)}; \]
by a monotone convergence argument, we will concern ourselves with the infinitary quantity. 

The proof is by contradiction: suppose that for some measure-preserving system, $(X,\mu,T)$, and some function $f: X \to \{|z| \leq 1\},$
\[ \mu(\{ \limsup_{M,N} |\mathbb{E}_{[M]} T^{a_n} f - \mathbb{E}_{[N]} T^{a_n} f| \gg \tau\}) \gg \tau \]
for some $1 \gg \tau > 0$. In this case, we could extract a finite subsequence of arbitrary length $K$ so that for each $1 \leq k \leq K$
\[ \mu(\{ \max_{M_{k-1} \leq M \leq M_k} |\mathbb{E}_{[M]} T^{a_n} f - \mathbb{E}_{[M_k]} T^{a_n} f| \gg \tau\}) \gg \tau;\]
there is no loss of generality here in assuming that 
\[ M_k,M \in \{ 2^{n/R} \}, \; \; \; R \gg \tau^{-1}\]
as whenever
\[ M = 2^{n/R} \leq M' < 2^{(n+1)/R} \]
are close,
\[ \mathbb{E}_{[M]} T^{a_n} f = \mathbb{E}_{[M']} T^{a_n} f + O(R^{-1}) = \mathbb{E}_{[M']} T^{a_n} f + o_{R \to \infty}(\tau); \]
we will implicitly restrict all times to this lacunary sequence. But, summing over $k \leq K$, we bound
\[ K \tau \ll \sum_{k \leq K} \mu(\{\max_{M_{k-1} \leq M \leq M_k} |\mathbb{E}_{[M]} T^{a_n} f - \mathbb{E}_{[M_k]} T^{a_n} f| \gg \tau\}) \leq C_{\tau}^{(X,\mu,T)},\]
which forces an absolute upper bound on $K \lesssim_\tau 1$, for the desired contradiction.
\end{proof}

So, Proposition \ref{p:pointwise} will be proven once we establish the following.

\begin{lemma}\label{l:trans}
There exists an absolute $c_0 > 0$ so that for each $\tau,H$, and each $(X,\mu,T)$
\[ C_{\tau,H}^{(X,\mu,T)} \lesssim \tau^{-1} C_{c_0 \tau,H}.\]
\end{lemma}
\begin{proof}
    By assumption, for any $|f| \leq 1$, and any $M_0<M_1< \dots < M_K$ if we set
    \begin{align}
        Z_k &:= Z_k(f,M_0,\dots,M_K,\tau) \\
        & \qquad := \{ w \in [H] : \max_{M_{k-1} \leq M \leq M_k} | f_M(w) - f_{M_k}(w)| \geq c_0 \tau\} \subset [H], \notag
    \end{align} 
    where all times $M \in \{  2^{n/R}  \}, R \gg \tau^{-1}$,
    then
    \begin{align}\label{e:Zptwise}
        \sum_{k \leq K} \frac{1}{H} \sum_{w \in [H]} \mathbf{1}_{Z_k}(w) &\leq \sum_{k \leq K \text{ good}} \frac{1}{H} \sum_{w \in [H]} \mathbf{1}_{Z_k}(w) + \sum_{k \leq K \text{ bad}} \frac{1}{H} \sum_{w \in [H]} \mathbf{1}_{Z_k}(w) \\
        & \qquad \leq c_0 \tau K + C_{c_0 \tau,H}; \notag
    \end{align}
above an index is \emph{bad} if 
\begin{align}\label{e:badcrit} |Z_k| \geq c_0 \tau H \end{align}
and \emph{good} otherwise.
    
Let $C_{\tau,H}^{(X,\mu,T)} \lesssim_H 1$ be as above; our job is to exhibit $c_0 > 0$ so that we may bound \[ C_{\tau,H}^{(X,\mu,T)} \lesssim \tau^{-1} C_{c_0 \tau,H}\]
independent of $H$ and $(X,\mu,T)$. If we set, for an appropriate  $f : X \to \{|z| \leq 1\}$,
\begin{align}
    U_k &:= U_k(f,M_0,\dots,M_K,\tau) \\
    & \qquad := \{ \max_{M_{k-1} \leq M \leq M_k} |\mathbb{E}_{[M]} T^{a_n} f - \mathbb{E}_{[M_k]} T^{a_n} f| \geq \tau \}, \; \; \; \mu(U_k) \geq \tau \notag
\end{align}  
for an appropriate realization of $C_{\tau,H}^{(X,\mu,T)} = K \lesssim_H 1$, then using the measure-preserving nature of $T$, we may bound
    \begin{align}\label{e:restI}
        \tau \cdot C_{\tau,H}^{(X,\mu,T)} = \tau K \leq \int_X \big( \sum_{k \leq K} \frac{1}{|I|} \sum_{w \in I} \mathbf{1}_{U_k}(T^w x) \big) \ d\mu(x), \; \; \; I := [H/10,H - H/10].
    \end{align}
We claim that $\mu$-a.e., we may dominate the integrand by a constant multiple of
\[ 
c_0 \tau K + C_{c_0 \tau,H}
\]
which leads to the desired bound
\[ C_{\tau,H}^{(X,\mu,T)} \lesssim \tau^{-1} C_{c_0 \tau,H}, \]
provided $c_0 > 0$ is sufficiently small.

To see this, let $x \in X$ be arbitrary, and define
\[ F(n) := T^{n} f(x) \cdot \mathbf{1}_{n \in [H]}; \]
the key observation is that for all $w \in I$ and $M \leq M_K \leq h^{-1}(H/100)$,
\[ \mathbb{E}_{[M]} T^{a_n} f(T^w x) = \mathbb{E}_{[M]} T^{a_n+w} f( x) \]
is precisely given by
\[ F_M(w) = \frac{1}{M} \sum_{m \in [M]} F(w+a_m) = \mathcal{I}\big(\mathbb{E}_{[M]} \mathcal{I}(F)(\cdot - a_m)\big)(w), \; \; \; \mathcal{I}(g) := \tilde{g}\]
so we may pointwise bound
\[ \sum_{k\leq K} \frac{1}{|I|} \sum_{w \in I} \mathbf{1}_{U_k}(T^w x) \lesssim \sum_{k \leq K} \frac{1}{H} \sum_{w \in [H]} \mathbf{1}_{Z_k}(w) \leq c_0 \tau K + C_{c_0\tau,H}, \]
for an appropriate choice of $\{ Z_k : k \leq K\}$, concluding the reduction.
\end{proof}

With this machinery in hand, to complete the proof of Theorem \ref{t:maindet} or \ref{t:mainran}, from which we have just seen that Theorem \ref{t:cor} derives, it suffices to observe that for any $\epsilon > 0$, $r \geq 2$, $\{ M_j \} \subset \mathbb{N}$, each operator
\[ \epsilon N_{\epsilon}(a_k)^{1/2}, \;  \mathcal{V}^r(a_k), \; \mathcal{O}_{\{M_j\}}(a_k) \]
satisfies the axioms of $\mathbf{N}(a_k)$ introduced via the triangle inequality for the latter two operators, and the inequalities
\begin{align}\label{e:triangle}&\epsilon {N}_\epsilon(\sum_{l=1}^L a_k^{(l)})^{1/2} \\
& \qquad \lesssim \min \big\{ L \sum_{l=1}^L \epsilon/L \cdot {N}_{\epsilon/L}(a_k^{(l)})^{1/2} , \sum_{l=1}^L l^2 \big( \frac{\epsilon}{10 l^2} \cdot {N}_{\frac{\epsilon}{10 l^2}}(a_k^{(l)})^{1/2} \big) \big\}, \end{align}
and
\begin{align}\label{e:homo} \epsilon N_{\epsilon}(\lambda a_k)^{1/2} = |\lambda| \cdot \big( \epsilon/|\lambda| \cdot N_{\epsilon/|\lambda|}( a_k)^{1/2} \big).\end{align}

This completes the proof.

\end{document}